\documentclass[ reqno]{amsart}

\usepackage{hyperref}

\newtheorem{theorem}{Theorem}[section]

\theoremstyle{definition}

\newtheorem{example}[theorem]{Example}

\theoremstyle{remark}
\newtheorem{remark}[theorem]{Remark}

\numberwithin{equation}{section}

\begin{document}

\title[Extremal Lipschitz extensions and Jensen's equations]{Explicit solutions of Jensen's auxiliary equations via extremal Lipschitz extensions}



\author[F. Charro]{Fernando Charro}
\address{Department of Mathematics, Wayne State University, Detroit, MI 48202, USA}

\email{fcharro@wayne.edu}
\thanks{Partially supported by MINECO grants MTM2016-80474-P
and MTM2017-84214-C2-1-P, Spain.
}



\date{}

\dedicatory{}

\commby{}

\begin{abstract}
In this note we prove that McShane and Whitney's Lipschitz extensions are viscosity solutions of Jensen's auxiliary equations,
known to have a key role in Jensen's celebrated proof of uniqueness of infinity harmonic functions, and therefore of Absolutely Minimizing Lipschitz Extensions.  To the best of the author's knowledge, this result does not appear to be known in the literature in spite of the vast amount of work around the~topic.
\end{abstract}

\maketitle

\section{Introduction}

Given a Lipschitz function $F:\partial\Omega\to\mathbb{R}$ with Lipschitz constant $\lambda$ one can consider the problem of finding a Lipschitz extension of the function to the interior of $\Omega$.  This problem has received great attention for many years, we refer the interested reader to \cite{aronsson2004tour} for a survey on the topic.

Notice that the best Lipschitz constant one can hope for the extension is $\lambda$ itself. This Lipschitz constant is achieved by the explicit extensions
\begin{equation}\label{McShane}
 \overline{u} (x)
 =
 \inf_{z\in\partial\Omega}
 \big(
  F(z)+\lambda|x-z|
 \big)
\end{equation}
and
\begin{equation}\label{Whitney}
 \underline{u}(x)
 =
  \sup_{z\in\partial\Omega}
 \big(
 F(z)-\lambda|x-z|
 \big) 
\end{equation}
 due to McShane \cite{mcshane1934} and Whitney \cite{Whitney1934}, respectively.  It is easy to see that $\overline{u}$, $\underline{u}$ coincide with $F$ at $\partial \Omega$ and are Lipschitz continuous with constant $\lambda$. 
 In fact, $\overline{u}=F$ on $\partial\Omega$ follows by noticing that 
for all $x\in\partial\Omega$, the
definition of $\overline{u}$ and the
 Lipschitz continuity of $F$ yield
\begin{equation} \label{proof.u=F}
\overline{u}(x)\leq F(x)\leq F(z)+\lambda\, |x-z|,\quad \textrm{for all}\ z\in \partial\Omega, 
\end{equation}
and similarly for $\underline{u}$.
On the other hand, the Lipschitz condition for $\overline{u}$  can be verified observing that if $x,y\in\mathbb{R}^n$,  then
\begin{equation}\label{proof.Lip.condition}
\overline{u}(x)\leq
\inf_{z\in\partial\Omega}
\big(
F(z)+\lambda(|y-z|+|x-y|)
\big)
=\overline{u}(y)+\lambda|x-y|,
\end{equation}
and then reversing the roles of $x,y$ (the case of $\underline{u}$ is similar).

 Furthermore, these extensions are extremal in the sense that any other Lipschitz extension $u$ satisfies 
\begin{equation}\label{AMLE.sandwhich}
\underline{u}\leq u\leq \overline{u}.
\end{equation}
To see this, notice that by the Lipchitz continuity of $u$,
\[
u(z)-\lambda\, |x-z|\leq u(x)\leq u(z)+\lambda\, |x-z|
\]
for all $x\in\mathbb{R}^n$ and $z\in \partial\Omega$ (note that $u(z)=F(z)$).

Whenever McShane and Whitney's Lipschitz extensions, $\underline{u}$ and $\overline{u}$ coincide, \eqref{AMLE.sandwhich} provides uniqueness and optimality of the extension. However, this rarely happens, see \cite{aronsson2004tour}.
Then, a natural question arises, how to find the ``best" extension of $F:\partial\Omega\to\mathbb{R}$ to the interior of $\Omega$. Or, in other words, how to find
$u$ with the least possible Lipschitz constant in every open set whose closure is compactly contained in $\Omega$. This extension exists and is unique, and is called an Absolutely Minimizing Lipschitz Extension (AMLE) following \cite{aronsson1967}.
It turns out that such AMLE is infinity harmonic (see \cite{aronsson2004tour, jensen1993}), i.e., it satisfies $-\Delta_\infty u=0$ in $\Omega$ in the viscosity sense, where,
\[
\Delta_\infty u(x)=
\left<D^2u(x)\, \nabla u(x),\nabla u(x)\right>
\]
is the well-known infinity Laplace operator (see  \cite{Lindqvist2014} for a survey of its applications).

In this note we prove that McShane and Whitney's extensions are viscosity solutions of Jensen's auxiliary equations,
known to have a key role in Jensen's celebrated proof of uniqueness of infinity harmonic functions (and hence of AMLE) in \cite{jensen1993}.  This question arose in connection with a modified Tug-of-War game studied in \cite{Anton.Charro.Wang} which models Jensen's auxiliary equations in graphs.
To the best of our knowledge, this result does not seem to be known in the literature in spite of the vast amount of work around the topic.

In the sequel, given  $g:K\subset\mathbb{R}^n \to \mathbb{R}$, Lipschitz continuous on $K$,  we will denote by $L_g(K)$  the smallest constant $\lambda\geq0$ for which $|g(x) - g(y)| \leq \lambda|x - y|$ for all $x, y \in K$. If $\lambda\geq L_g(K) $, then we will say that $\lambda$ is ``a Lipschitz constant for $g$".

The main result of the paper is the following.

\begin{theorem}\label{explicit.sols.Jensens.eqs}
Let  $F:\partial\Omega\to\mathbb{R}$ be a Lipschitz function with least Lipschitz constant $L_F(\partial\Omega)$. Then, for every $\lambda \geq L_F(\partial\Omega)$, McShane's extension
$\overline{u}$ defined in  \eqref{McShane} is the unique  viscosity solution of 
\begin{equation}\label{Jenesen.min.equation}
\begin{cases}
\min\left\{|\nabla u(x)|-\lambda,-\Delta_\infty u(x)\right\}=0 & \textrm{in}\ \Omega\\
u(x)=F(x) & \textrm{on}\ \partial\Omega.
\end{cases}
\end{equation}
Similarly, Whitney's extension $\underline{u}$ defined in \eqref{Whitney} is the unique viscosity solution of
\begin{equation}\label{Jenesen.max.equation}
\begin{cases}
\max\left\{\lambda-|\nabla u(x)|,-\Delta_\infty u(x)\right\}=0  & \textrm{in}\ \Omega\\
u(x)=F(x) & \textrm{on}\ \partial\Omega.
\end{cases}
\end{equation} 
On the other hand, whenever $\lambda < L_F(\partial\Omega)$, the functions $\overline{u},\underline{u}$ still satisfy the equations in \eqref{Jenesen.min.equation} and \eqref{Jenesen.max.equation} in the interior of $\Omega$ but fail to achieve the boundary condition $u=F$ on $\partial\Omega$.
\end{theorem}

As a motivation, we have the following example.
\begin{example}
Let $\lambda>0,$ $\Omega\subset\mathbb{R}^n$ and consider $u_\lambda(x)=\lambda\,\textrm{dist}(x,\partial\Omega)$.
 It can be checked by direct computation  that $u_\lambda$
 is the  unique viscosity solution to
\[
\left\{
\begin{split}
&\min\{|\nabla u|-\lambda,-\Delta_\infty u\}=0\quad\text{in}\ \Omega,\\
&u=0\quad\text{on}\ \partial\Omega.
\end{split}
\right.
\]
This agrees with Theorem \ref{explicit.sols.Jensens.eqs} since for every $\lambda\geq0=L_F(\partial\Omega)$ we have
\[
 \overline{u} (x)
 =
 \lambda
  \inf_{z\in\partial\Omega}
|x-z|
=
\lambda\,\textrm{dist}(x,\partial\Omega).
\]
\end{example}

The fact that an AMLE is infinity harmonic (again, see \cite{aronsson2004tour, jensen1993}) makes it  a subsolution of
\eqref{Jenesen.min.equation} 
  and a supersolution of \eqref{Jenesen.max.equation}, respectively.
  Then, the comparison principle for Jensen's equations \eqref{Jenesen.min.equation} and \eqref{Jenesen.max.equation} (see  \cite[Theorems 2.1 and 2.15]{jensen1993}) offers another perspective on    \eqref{AMLE.sandwhich}, which follows by comparison.  
 In the next result we show that this is a general fact that does not depend on the infinity-harmonicity of the AMLE, i.e., we prove that any Lipschitz extension
 is a subsolution of \eqref{Jenesen.min.equation}  and a supersolution of \eqref{Jenesen.max.equation}, respectively.

\begin{theorem}\label{second.theorem}
 Let  $F:\partial\Omega\to\mathbb{R}$ be  Lipschitz continuous, and let $u$ be any Lipschitz extension of $F$ to $\Omega$, i.e., a Lipschitz function $u:\Omega\to\mathbb{R}$ such that 
 $u=F$ on $\partial\Omega$ and has Lipschitz constant $L_u(\Omega)=L_F(\partial\Omega)$. Then, for every $\lambda\geq L_F(\partial\Omega)$
\begin{equation}\label{Jenesen.min.equation.u}
\begin{cases}
\min\left\{|\nabla u(x)|-\lambda,-\Delta_\infty u(x)\right\}\leq0 & \textrm{in}\ \Omega\\
u(x)=F(x) & \textrm{on}\ \partial\Omega.
\end{cases}
\end{equation}
and 
\begin{equation}\label{Jenesen.max.equation.u}
\begin{cases}
\max\left\{\lambda-|\nabla u(x)|,-\Delta_\infty u(x)\right\}\geq0 & \textrm{in}\ \Omega\\
u(x)=F(x) & \textrm{on}\ \partial\Omega.
\end{cases}
\end{equation} 
in the viscosity sense.
\end{theorem}

This can also be understood in view of Rademacher's Theorem: A Lipschitz function $u$ on an open subset of the Euclidean space is differentiable almost everywhere  and the number $\|\nabla u\|_\infty$ is bounded from above by
the Lipschitz constant of $u$ (if in addition the domain is convex, then the least Lipschitz constant equals $\|\nabla u\|_\infty$).

\begin{remark}
Theorems \ref{explicit.sols.Jensens.eqs} and \ref{second.theorem} also hold with $ \Delta_\infty^N u $ in place of $ \Delta_\infty u $, where
\begin{equation}\label{infinityLaplacian_normalized}
\Delta_\infty^N u(x):=
\begin{cases}
\left<D^2u(x)\, \frac{\nabla u(x)}{|\nabla u(x)|},\frac{\nabla u(x)}{|\nabla u(x)|}\right>, & \textnormal{if}\ \nabla u(x)\neq0 \\[0.5em]
\lim_{y\to x}\frac{2(u(y)-u(x))}{|y-x|^2}, & \textnormal{otherwise}
\end{cases}
\end{equation}
is the normalized infinity Laplacian, well known for its role in the modeling of random Tug-of-War games, see   \cite{Lindqvist2014} and the references therein.
\end{remark}

We would like to finish this introduction  pointing out that the Taylor expansion arguments in the proof of Theorem \ref{explicit.sols.Jensens.eqs} have an interesting connection with the numerical analysis of equations  \eqref{Jenesen.min.equation} and \eqref{Jenesen.max.equation}.
More precisely, equations \eqref{Jenesen.min.equation} and \eqref{Jenesen.max.equation} can be respectively approximated by 
\begin{equation}\label{disc.scheme.lower}
\begin{split}
\min\Bigg\{
\frac{1}{\epsilon}&\left(
u(x)-\inf_{y\in\overline{B}_\epsilon(x)\cap\overline{\Omega}}u(y)-\epsilon \lambda
\right),\\
&\hspace{50pt}\frac{1}{\epsilon^2}\left(2 u(x)-\sup_{y\in\overline{B}_\epsilon(x)\cap\overline{\Omega}}u(y)-\inf_{y\in\overline{B}_\epsilon(x)\cap\overline{\Omega}}u(y)
\right)\Bigg\}=0
\end{split}
\end{equation}
and
\begin{equation}\label{disc.scheme.upper}
\begin{split}
\max\Bigg\{
\frac{1}{\epsilon}&\left(
u(x)-\sup_{y\in\overline{B}_\epsilon(x)\cap\overline{\Omega}}u(y)+\epsilon \lambda
\right),\\
&\hspace{50pt}\frac{1}{\epsilon^2}\left(2 u(x)-\sup_{y\in\overline{B}_\epsilon(x)\cap\overline{\Omega}}u(y)-\inf_{y\in\overline{B}_\epsilon(x)\cap\overline{\Omega}}u(y)
\right)\Bigg\}=0,
\end{split}
\end{equation}
which can be regarded as discrete elliptic schemes in the sense of \cite{Oberman2006} (and, therefore, monotone in the sense of \cite{barles1991}). 

Moreover, in a similar way to the Taylor expansion arguments in the proof of Theorem \ref{explicit.sols.Jensens.eqs}, one can  show that schemes \eqref{disc.scheme.lower} and \eqref{disc.scheme.upper} are   consistent (see \cite[Section 2]{barles1991} for the  definition). This means, roughly speaking, that the finite-difference operator converges in the viscosity sense towards the continuous operator of the PDE  as $\epsilon\to0$. 
 Monotonicity and consistency, altogether with stability are important requirements for convergence, as established in the seminal paper \cite{barles1991}. Informally, the authors in \cite{barles1991} prove that any monotone, stable, and consistent scheme converges  provided that the limiting equation satisfies a type of comparison principle known as ``strong uniqueness property", which is usually difficult to prove.

It seems an interesting question to tackle the convergence of schemes \eqref{disc.scheme.lower} and \eqref{disc.scheme.upper} and their numerical implementation; however,  we will not discuss that problem here.

\section{Proofs of Theorems \ref{explicit.sols.Jensens.eqs} and \ref{second.theorem}}

We proceed first to prove Theorem \ref{second.theorem}.
\begin{proof}[Proof of Theorem \ref{second.theorem}]
 Let us prove the result for \eqref{Jenesen.min.equation.u} since the proof for \eqref{Jenesen.max.equation.u} is similar. Let $\hat{x}\in\Omega$ and  $\phi\in C^2(\Omega)$ such that $\phi$ touches $u$ at $\hat{x}$  from above in a neighborhood of $\hat{x}$. Our goal is to prove
\begin{equation}\label{goal.subsol.Lip.exts.general.u} 
\min\left\{|\nabla \phi(\hat{x})|-\lambda,-\Delta_\infty \phi(\hat{x})\right\}\leq0. 
\end{equation}
Notice that we can assume $\nabla \phi(\hat{x})\neq0$ since we are done otherwise. Then, the contact condition and a Taylor expansion yield
\[
u(x)
\leq
\phi(x) = u (\hat{x})+\langle\nabla\phi(\hat{x}), x-\hat{x}\rangle+o(|x-\hat{x}|)
\quad
\textrm{as} 
\ x\to\hat{x}
\]
 Choose $x=\hat{x}-\alpha \nabla\phi(\hat{x})$, with $\alpha>0$ small enough. Then
\[
-\lambda\,\alpha|\nabla\phi(\hat{x})|
\leq
 u\big(\hat{x}-\alpha \nabla\phi(\hat{x})\big)- u (\hat{x})\leq-\alpha |\nabla\phi(\hat{x})|^2+o(\alpha)
\] 
 by the Lipschitz continuity of $u$. Dividing both sides by $-\alpha|\nabla\phi(\hat{x})|$ and letting $\alpha\to0$, we get $|\nabla\phi(\hat{x})|\leq \lambda$ as desired.
 \end{proof}

We present now  the proof of Theorem \ref{explicit.sols.Jensens.eqs}.

\begin{proof}[Proof of Theorem \ref{explicit.sols.Jensens.eqs}]
Assume first that $\lambda \geq L_F(\partial\Omega)$, and 
let us prove that $\overline{u}$ is a viscosity solution of \eqref{Jenesen.min.equation}. First, we will show the supersolution case. Observe that for every $z\in\partial\Omega$, the cone $C(x)=F(z)+\lambda|x-z|$ satisfies
 \[
 \min\left\{|\nabla C(x)|-\lambda,-\Delta_\infty  C(x)\right\}=0\quad \textrm{in}\ \Omega,
 \]
 in the classical sense,
 and therefore $\overline{u}$ is a viscosity supersolution in $\Omega$ because it is an infimum of supersolutions. Moreover, $\overline{u}=F$, as discussed in \eqref{proof.u=F}.

Alternatively,  let $\hat{x}\in\Omega$ and  $\phi\in C^2(\Omega)$ such that $\phi$ touches $\overline{u}$ at $\hat{x}$  from below in a neighborhood of $\hat{x}$. Our goal is to prove that
\begin{equation}\label{goal.limitPDE.supersol.Lip.ext.proof} 
\min\left\{|\nabla \phi(\hat{x})|-\lambda,-\Delta_\infty  \phi(\hat{x})\right\}\geq0. 
\end{equation}
Notice that by the Lipschitz continuity of $F$, the function $z\mapsto F(z)+\lambda|x-z|$ is continuous for each fixed $x$, and we have that
\[
\phi(\hat{x})=\overline{u}(\hat{x})
=
 \min_{z\in\partial\Omega}
 \big(
  F(z)+\lambda|\hat{x}-z|
 \big)
=
F(\hat{z})+\lambda|\hat{x}-\hat{z}|
\]
for some $\hat{z}\in\partial\Omega$. On the other hand, 
\[
\phi(x)\leq\overline{u}(x)
\leq F(\hat{z})+\lambda|x-\hat{z}|
\]
and we find that $\phi$ touches the cone $C(x)=F(\hat{z})+\lambda|x-\hat{z}|$ at $\hat{x}$  from below in a neighborhood of $\hat{x}$. Then, $\nabla \phi(\hat{x})=\nabla C(\hat{x})$ and $D^2 \phi(\hat{x})\leq D^2 C(\hat{x})$ and we deduce 
\[
-\Delta_\infty  \phi(\hat{x})\geq -\Delta_\infty  C(\hat{x})=0,
\qquad
\textrm{and}\qquad
|\nabla \phi(\hat{x})|=|\nabla C(\hat{x})|=\lambda,
\]
which, yield \eqref{goal.limitPDE.supersol.Lip.ext.proof}.

We proceed now to prove that $\overline{u}$ is a viscosity subsolution of \eqref{Jenesen.min.equation}. Notice that we can apply Theorem \ref{second.theorem}. However, we are going to show a different argument which shows an interesting
 connection with the numerical analysis of equations  \eqref{Jenesen.min.equation} and \eqref{Jenesen.max.equation}.

To this aim, let $\hat{x}\in\Omega$ and  $\phi\in C^2(\Omega)$ such that $\phi$ touches $\overline{u}$ at $\hat{x}$  from above in a neighborhood of $\hat{x}$. Our goal is to prove
\begin{equation}\label{goal.subsol.Lip.exts} 
\min\left\{|\nabla \phi(\hat{x})|-\lambda,-\Delta_\infty  \phi(\hat{x})\right\}\leq0. 
\end{equation}
By the continuity of $\overline{u}$ (see  \eqref{proof.Lip.condition}), for $\epsilon$ small enough we can write
\[
\begin{split}
\min_{x\in \overline{B}_\epsilon(\hat{x})}\overline{u}(x)
&=
\min_{x\in \overline{B}_\epsilon(\hat{x})}
 \inf_{z\in\partial\Omega}
 \big(
  F(z)+\lambda|x-z|
 \big)
 \\
&\geq
 \inf_{z\in\partial\Omega}
 \big(
  F(z)+\lambda|\hat{x}-z|-\epsilon \lambda
 \big)
   =
  \overline{u}(\hat{x})-\epsilon \lambda,
\end{split}
\]
where we have used that $|\hat{x}-z|\leq\epsilon+|x-z|$ for every $x\in \overline{B}_\epsilon(\hat{x})$. Therefore, 
\[
 \frac{1}{\epsilon}\bigg(
\phi(\hat{x})-\min_{x\in\overline{B}_\epsilon(\hat{x})}\phi(x)
\bigg)
\leq
 \frac{1}{\epsilon}\bigg(
 \overline{u}(\hat{x})-\min_{\overline{B}_\epsilon(\hat{x})}\overline{u}
\bigg)
\leq
\lambda.
\]
We claim that
\begin{equation}\label{varphi.at.minimum.Lip.exts}
\min_{x\in\overline{B}_\epsilon(\hat{x})}\phi(x)
=
\phi
\left(
\hat{x}-\epsilon\left[\frac{\nabla\phi(\hat{x})}{|\nabla\phi(\hat{x})|}+o(1)\right]
\right)
\quad\textnormal{as}\ \ \epsilon\to0.
\end{equation}
Then,  a first-order Taylor expansion yields
\[
\frac{1}{\epsilon}\bigg(
\phi(\hat{x})-\min_{x\in\overline{B}_\epsilon(\hat{x})}\phi(x)
\bigg)
=
|\nabla \phi(\hat{x})| +o(1) 
\quad\textnormal{as}\ \ \epsilon\to0
\]
and we deduce $|\nabla\phi(\hat x)|\leq \lambda$  and, hence, that \eqref{goal.subsol.Lip.exts} holds.

We proceed to prove  claim \eqref{varphi.at.minimum.Lip.exts} for the sake of completeness. Notice that we can assume $\nabla\phi(\hat x)\neq0$ since otherwise $|\nabla\phi(\hat x)|\leq \lambda$ holds  and there is nothing to prove. Write 
\[
\min_{x\in\overline{B}_\epsilon(\hat{x})}\phi(x)=\phi(\hat{x}-\epsilon v_\epsilon)
\]
 for some $v_\epsilon\in \overline{B}_1(0)$. 
Observe that  $|v_\epsilon|=1$ for every $\epsilon$ small enough because, otherwise,  there would be a subsequence  $\hat{x}-\epsilon_k v_{\epsilon_k}$ of
 interior minimum points of $\phi$ in $B_{\epsilon_k}(\hat{x})$ for which 
 $\nabla\phi(\hat{x}-\epsilon_k v_{\epsilon_k})=0$, a contradiction as $\epsilon_k\to0$.

It remains to show that, actually,
\begin{equation}\label{minimum_point.Lip.exts}
v_\epsilon= \frac{\nabla\phi(\hat{x})}{|\nabla\phi(\hat{x})|}+o(1)
\quad\textnormal{as}\ \ \epsilon\to0.
\end{equation}
Let $\omega$ be any fixed direction with $|\omega|=1$. Then, 
\[
\phi(\hat{x}-\epsilon v_\epsilon)=\min_{x\in\overline{B}_\epsilon(\hat{x})}\phi(x)\leq\phi(\hat{x}-\epsilon\, \omega),
\]
and  a Taylor expansion of $\phi$ around $\hat{x}$ gives 
\begin{equation*}
\left<\nabla\phi(\hat{x}),\, v_\epsilon\right>+o(1)\geq\frac{-\phi(\hat{x}-\epsilon\, \omega)+\phi(\hat{x})}{\epsilon}=\left<\nabla\phi(\hat{x}),\, \omega\right>+o(1)\quad\textnormal{as}\ \ \epsilon\to0.
\end{equation*}
Since the previous argument holds for any direction $\omega$, we  have    \eqref{minimum_point.Lip.exts} as desired.

The proof that  $\underline{u}$ is a viscosity solution of \eqref{Jenesen.max.equation} is similar.

To conclude, let us point out that in the case $\lambda < L_F(\partial\Omega)$ we can follow the argument above  and show that
the functions $\overline{u}$, $\underline{u}$ respectively satisfy the equations in \eqref{Jenesen.min.equation} and \eqref{Jenesen.max.equation} in the interior of $\Omega$. In fact, \eqref{McShane}, \eqref{Whitney} are still Lipschitz continuous with constant $\lambda$ in the interior of $\Omega$ by \eqref{proof.Lip.condition}. However, \eqref{proof.u=F} does not work and we can only say
 $\overline{u}\leq F\leq \underline{u}$ on $\partial\Omega$ (which holds by definition) and  $\overline{u}$, $\underline{u}$ fail to achieve the boundary condition.
\end{proof}

\bibliographystyle{amsplain}
\bibliography{references}

\providecommand{\bysame}{\leavevmode\hbox to3em{\hrulefill}\thinspace}
\providecommand{\MR}{\relax\ifhmode\unskip\space\fi MR }
\providecommand{\MRhref}[2]{%
  \href{http://www.ams.org/mathscinet-getitem?mr=#1}{#2}
}
\providecommand{\href}[2]{#2}
\begin{thebibliography}{1}

\bibitem{Anton.Charro.Wang}
Marcos Ant\'on, Fernando Charro, and Pei-Yong Wang, \emph{Totalitarian random
  tug-of-war games in graphs}, Comm. on Stochastic Analysis \textbf{13} (2019),
  no.~3.

\bibitem{aronsson1967}
Gunnar Aronsson, \emph{Extension of functions satisfying lipschitz conditions},
  Ark. Mat. \textbf{6} (1967), no.~6, 551--561.

\bibitem{aronsson2004tour}
Gunnar Aronsson, Michael Crandall, and Petri Juutinen, \emph{A tour of the
  theory of absolutely minimizing functions}, Bulletin of the American
  mathematical society \textbf{41} (2004), no.~4, 439--505.

\bibitem{barles1991}
Guy Barles and Panagiotis~E. Souganidis, \emph{{Convergence of approximation
  schemes for fully nonlinear second order equations}}, Asymptotic analysis
  \textbf{4} (1991), 271--283.

\bibitem{jensen1993}
Robert Jensen, \emph{{Uniqueness of Lipschitz extensions: minimizing the sup
  norm of the gradient}}, Archive for Rational Mechanics and Analysis
  \textbf{123} (1993), 51--74.

\bibitem{Lindqvist2014}
Peter Lindqvist, \emph{Notes on the infinity laplace equation}, Springer, 2016.

\bibitem{mcshane1934}
E.~J. McShane, \emph{Extension of range of functions}, Bull. Amer. Math. Soc.
  \textbf{40} (1934), no.~12, 837--842.

\bibitem{Oberman2006}
Adam~M. Oberman, \emph{{Convergent difference schemes for degenerate elliptic
  and parabolic equations: Hamilton--Jacobi equations and free boundary
  problems}}, SIAM Journal on Numerical Analysis \textbf{44} (2006), 879--895.

\bibitem{Whitney1934}
Hassler Whitney, \emph{Analytic extensions of differentiable functions defined
  in closed sets}, Transactions of the American Mathematical Society
  \textbf{36} (1934), no.~1, 63--89.

\end{thebibliography}

\end{document}